\documentclass[reqno,11pt]{amsart}
\usepackage{amsmath,amssymb,amsthm, verbatim}
\usepackage{url}
\usepackage[usenames, dvipsnames]{color}
\usepackage[foot]{amsaddr}

\usepackage[letterpaper,hmargin=1in,vmargin=1in]{geometry}

\usepackage{graphicx}
\usepackage{enumerate,pinlabel}
\usepackage{mathrsfs,graphicx,url}
\usepackage[usenames,dvipsnames]{xcolor}
\usepackage[colorlinks=true,linkcolor=Black,citecolor=Black, urlcolor=Black]{hyperref}
\numberwithin{equation}{section}

\theoremstyle{plain}
\newtheorem{theorem}{Theorem}[section]

\newtheorem*{theorem*}{Theorem}
\newtheorem*{lemma*}{Lemma}
\newtheorem{lemma}[theorem]{Lemma}
\newtheorem{proposition}[theorem]{Proposition}

\theoremstyle{definition}

\theoremstyle{remark}
\newtheorem{remark}[theorem]{Remark}

\newcommand{\e}{\epsilon}
\newcommand{\ep}{\varepsilon}

\newcommand{\R}{\mathbb{R}}
\newcommand{\US}{\mathbb{S}}

\newcommand\supp{\mathop{\rm supp}}

\newcommand\real{\mathop{\rm Re}}
\newcommand\imag{\mathop{\rm Im}}

\newcommand*{\defeq}{\mathrel{\vcenter{\baselineskip0.5ex \lineskiplimit0pt

                     \hbox{\scriptsize.}\hbox{\scriptsize.}}}%
                     =}

\begin{document}
\title[Resolvent bounds for short range $L^\infty$ potentials]{Semiclassical resolvent bounds for short range $L^\infty$ potentials with singularities at the origin}

\author{Jacob Shapiro}
\address{Department of Mathematics, University of Dayton, Dayton, OH 45469-2316, USA}
\email{jshapiro1@udayton.edu}

\begin{abstract}
We consider, for $h, E > 0$, resolvent estimates for the semiclassical Schr\"odinger operator $-h^2 \Delta + V - E$. Near infinity, the potential takes the form $V = V_L+ V_S$, where $V_L$ is a long range potential which is Lipschitz with respect to the radial variable, while $V_S = O(|x|^{-1} (\log |x|)^{-\rho})$ for some $\rho > 1$. Near the origin, $|V|$ may behave like $|x|^{-\beta}$, provided $0 \le \beta < 2(\sqrt{3} -1)$. We find that, for any $\tilde{\rho} > 1$, there are $C, \, h_0 >0$ such that we have a resolvent bound of the form  $\exp(Ch^{-2} (\log(h^{-1}))^{1 + \tilde{\rho}})$ for all $h \in (0, h_0]$. The $h$-dependence of the bound improves if $V_S$ decays at a faster rate toward infinity.
\end{abstract}

\maketitle

\section{Introduction and statement of results}
Let $\Delta \defeq \sum_{j =1}^n \partial^2_j \le 0$ be the Laplacian on $\mathbb{R}^n$, $n \ge 2$. In this article, we study the semiclassical Schr\"odinger operator with real valued potential,
\begin{equation} \label{P}
 P = P(h) \defeq -h^2 \Delta + V(x) : L^2(\mathbb{R}^n) \to L^2(\mathbb{R}^n),\qquad h \in (0,1), \, x \in \R^n.
\end{equation}
We use $(r, \theta) = (|x|, x/|x|) \in (0, \infty) \times \US^{n-1}$ to denote polar coordinates on $\R^n \setminus \{0\}$. For a function $f$ defined on some subset of $\R^n$, we use the notation $f(r, \theta) \defeq f(r \theta)$ and denote the derivative with respect to the radial variable by $f' \defeq \partial_r f.$

We first describe the conditions we impose on the potential $V$. Let $\chi \in C^\infty([0, \infty); [0,1])$ be such that $\chi = 1$ near $[0,1]$ while $\chi = 0$ near $[2, \infty)$. We suppose that 
\begin{equation} \label{V0}
V_0 \defeq \chi V \in L^p(\R^n), \qquad  \text{for some } p \ge 2, \, p > n /2,
\end{equation} 
has the bound
\begin{equation} 
|V_0(r, \theta)|\le c_0 r^{-\beta}, \label{V0 bd}
\end{equation}
for some $c_0 > 0$ and some 
\begin{equation} \label{beta}
 0 \le \beta < 2 (\sqrt{3} - 1) \approx 1.464.
\end{equation}

On the other hand, we suppose $(1- \chi)V$ may be decomposed  as a
sum of long- and short-range terms:
\begin{equation} \label{L plus S}
(1- \chi)V = V_L + V_S , \qquad V_L, \, V_S \in L^\infty(\R^n).
\end{equation}
The long-range term $V_L$ must satisfy, for some $c_L  > 0$ and some
\begin{equation} \label{y}
y: [1, \infty) \to [0,1], \qquad \lim_{r \to \infty} y(r) = 0, 
\end{equation}
that
\begin{equation}
V_L(r, \theta) \mathbf{1}_{r \ge 1} \le c_L y(r), \label{VL r big}
\end{equation}
where $\mathbf{1}_{r \ge 1}$ denotes the characteristic function of $\{x \in \R^n : |x| = r \ge 1 \}$. We also require that there is a function $V'_L \in L^1_{\text{loc}}(\R^n \setminus \{0\})$ such that, for each $\theta \in \US^{n-1}$, the function $(0, \infty) \ni r  \mapsto V_L(r, \theta)$ has distributional derivative equal to $r \mapsto V_L'( r ,\theta)$, and
\begin{equation}
V_L'(r, \theta) \mathbf{1}_{r \ge 1} \le c_L r^{-1} m_L(r). \label{VL prime r big}
\end{equation}
where $m_L(r) : [1, \infty) \to (0, 1]$ has the properties
\begin{equation} \label{mL}
\lim_{r \to \infty} m_L(r) = 0, \qquad  r^{-1} m_L(r)  \in L^1[1, \infty). 
\end{equation} 
A typical example of the function $m_L$ is $m_L(r) = (\log r + 1)^{-\rho}$ for some $\rho > 1.$

As for the short-range term $V_S$, we require
\begin{equation} 
|V_S(r, \theta) | \mathbf{1}_{r \ge 1} \leq  c_S m_S(r) r^{-1- \delta}, \label{VS bound}
\end{equation}
for some $c_S > 0$ and $0 \le \delta \le 1$. Depending on the value of $\delta$, $m_S : [1, \infty) \to [0, 1]$ should satisfy
\begin{equation} \label{delta}
\begin{aligned}
&r^{-1} m^2_S(r) \in L^1[1,\infty) && \delta = 1,  \\
&m_S(r) = 1  && 0 < \delta < 1,\\
&m_S(r) = (\log r + 1)^{-\rho} \text{ for some $\rho > 1$}  && \delta = 0. 
\end{aligned}
\end{equation}

The properties \eqref{V0} and \eqref{L plus S} imply $V \in L^p(\R^n ; \R) + L^\infty(\R^n ; \R)$ for some $p \ge 2, \, p > n /2$. Therefore, by \cite[Theorem 8]{ne64}, $P$ is self-adjoint $L^2(\R^n) \to L^2(\R^n)$ when equipped with domain the Sobolev space $H^2(\mathbb{R}^n)$. Thus the resolvent $(P - z)^{-1}$ is bounded $L^2(\R^n) \to L^2(\R^n)$ for all $z \in \mathbb{C} \setminus \mathbb{R}$. Our main result is the following limiting absorption resolvent estimate.
\begin{theorem} \label{Linfty thm}
Let $n \ge 2$. Fix $s > 1/2$ and $[E_{\min} , E_{\max}] \subseteq (0, \infty)$. Suppose $V$ satisfies properties \eqref{V0} through \eqref{delta}. Define
\begin{equation} \label{defn g}
    g^\pm_{s}(h, \varepsilon) \defeq \|\langle x \rangle^{-s} (P(h) - E \pm i\varepsilon)^{-1}  \langle x \rangle^{-s}\|_{L^2(\R^n) \to L^2(\R^n)}, \qquad  \ep, \, h > 0,
\end{equation}
where $\langle x \rangle = \langle r \rangle \defeq (1 + r^2)^{1/2}$.

If $\delta = 1$, there exist $C > 0$ and $h_\delta \in (0,1)$ independent of $\ep$ and $h$ so that  
\begin{equation} \label{Linfty resolv est delta 1}
g^{\pm}_s(h,\ep)\leq  \exp \big( C h^{-\frac{4}{3}} \log(h^{-1})\big), \qquad E \in [E_{\min} , E_{\max}],\, h \in (0,h_\delta], \, \ep > 0.
\end{equation}
If $0< \delta < 1$, then for any $\e > 0$,  there exist $C > 0$ and $h_\delta \in (0,1)$ independent of $\ep$ and $h$ so that
\begin{equation} \label{Linfty resolv est delta between 0 1}
g^{\pm}_s(h,\ep)\leq  \exp \big( C h^{- \frac{2\delta + 2}{2\delta + 1} - \epsilon}) , \qquad E \in [E_{\min} , E_{\max}], \, h \in (0,h_\delta], \, \ep > 0.
\end{equation}
Finally, if $\delta = 0$, then for any $\tilde{\rho} \in (1, \rho]$, there exist $C  >0$ and $h_\delta \in (0,1)$ independent of $\ep$ and $h$ so that
\begin{equation} \label{Linfty resolv est delta 0}
g^{\pm}_s(h,\ep)\leq  \exp \big( C h^{-2} (\log(h^{-1}))^{1 + \tilde{\rho}} \big), \qquad E \in [E_{\min} , E_{\max}], \, h \in (0,h_\delta], \, \ep > 0.
\end{equation}
\end{theorem}
\begin{remark} The proof of Theorem \ref{Linfty thm} in fact establishes a more complicated but slightly improved version of \eqref{Linfty resolv est delta between 0 1}. For any $\e > 0$, there exist $C >0$ and $h_\delta \in (0,1)$ independent of $\ep$ and $h$ so that
\begin{equation*}
g^{\pm}_s(h,\ep)\leq  \exp \big( C h^{- \frac{4}{3} - \frac{2(1 - \delta) + \lambda^{-1}}{3(1 + 2\delta - \lambda^{-1})}} (\log(h^{-1}))^{1 + \e} \big), \qquad E \in [E_{\min} , E_{\max}], \, h \in (0,h_\delta], \, \ep > 0,
\end{equation*}
where $\lambda = \log (\log(h^{-1}))$. 
\end{remark}

\begin{remark}
The condition $\tilde{\rho} \in (1, \rho]$ is needed for technical reasons in the proof of \eqref{Linfty resolv est delta 0}. However, it is clear that once  \eqref{Linfty resolv est delta 0} holds for some $\tilde{\rho} \in (1, \rho]$, it holds for all $\tilde{\rho} > \rho$ too (with the same constants $C$ and $h_\delta$).
\end{remark}

Theorem \ref{Linfty thm} improves upon recent work on resolvent estimates in low regularity in several ways. When $\delta = 1$, the bound \eqref{Linfty resolv est delta 1} was previously proved in \cite{gash22b} if $n \ge 3$, $V_S = O((r + 1)^{-1}m_S(r))$, and $V_L = 0$. Second, when $0 < \delta < 1$, it was established in \cite{vo20b} that, if $n \ge 3$, $V_S = O((r + 1)^{-1-\delta})$, and $V_L'(r, \theta) = O((r + 1)^{-1-\gamma})$ for some $\gamma > 0$, then
\begin{equation*}
g^{\pm}_s(h,\ep)\le \exp \big( Ch^{ -\frac{2}{3} - \max \big\{ \frac{7}{3\delta }, \frac{4}{3 \gamma } \big\}} (\log(h^{-1}))^{\max \big\{ \frac{1}{\delta}, \frac{1}{\gamma} \big\}} \big).
\end{equation*}
Thus, the novelties of Theorem \ref{Linfty thm} are that it gives resolvent bounds for more types of decay conditions on $V_L$ and $V_S$, improves those bounds in several cases, allows $V$ to be singular as $r \to 0$, and includes the dimension two case. 

Furthermore, Theorem \ref{Linfty thm} warrants comparison with the resolvent bounds obtained in \cite{vo21, vo22} for short-range, \textit{radially symmetric} $L^\infty$ potentials $V$: 
\begin{equation*}
g^{\pm}_s(h,\ep) \le \begin{cases} 
\exp \big(C h^{\frac{\delta +1}{\delta}}  (\log(h^{-1}))^{\frac{1}{\delta}} \big) & V = O((r + 1)^{-1 -\delta}),\, \delta > 3,\\
\exp \big(C h^{-\frac{4}{3}} \big) & V = O((r + 1)^{-1 -\delta}),\,  1 < \delta \le 3, \\
\exp \big( Ch^{-\frac{2\delta+ 2}{2\delta + 1}} \big( \log(h^{-1}) \big)^{\frac{\delta + 2}{2 \delta  + 1}} \big) & V = O((r + 1)^{-1 -\delta}), \, 0 < \delta \le 1,  \\
\exp \big( Ch^{-2}\big) &  V = O((r +1)^{-1} \log(r + 2)^{-\rho}),\, \rho > 1.
\end{cases}
\end{equation*}
Thus, another way to interpret how Theorem \ref{Linfty thm} extends the previous literature is that it shows arbitrary short-range potentials have resolvent bounds similar to those for short-range radial potentials, though additional losses remain. 

Bounds on $g^{\pm}_s$ are known to hold under various geometric, regularity, and decay assumptions. Burq \cite{bu98, bu02} showed  $g^{\pm}_s \le e^{Ch^{-1}}$ for $V$ smooth and decaying sufficiently fast near infinity, and also for more general perturbations of the Laplacian. Cardoso and Vodev \cite{cavo02} extended Burq's estimate to infinite volume Riemannian manifolds which may contain cusps. This exponential behavior is sharp in general, see \cite{ddz15} for exponential resolvent lower bounds. On $\R^n$, $n \ge 2$, $g^{\pm}_s \le e^{Ch^{-1}}$ still holds if $V$ has long-range decay and Lipschitz regularity with respect to the radial variable \cite{da14, sh19, vo20c, gash22a, ob23}. Potentials with singularities near zero are treated in \cite{gash22a, ob23}, and in particular \cite{ob23} requires $\partial_rV(r, \theta)\mathbf{1}_{r \le 1} = O(r^{-j - \tilde{\beta}})$, for
\begin{equation} \label{beta tilde}
 0 < \tilde{\beta} < 4(\sqrt{2} -1) \approx 1.657
\end{equation}
 and  $j = 0,\, 1$. In one dimension, $g^{\pm}_s \le e^{Ch^{-1}}$ if $V$ is a finite Borel measure \cite{lash23}.

In contrast, if $V : \R^n \to \R$, $n \ge 2$, has purely $L^\infty$ terms, it is an open problem to determine whether the bounds \eqref{Linfty resolv est delta 1},  \eqref{Linfty resolv est delta between 0 1}, and \eqref{Linfty resolv est delta 0} have optimal $h$-dependence. Further works on resolvent estimates with little regularity assumed are \cite{vo14, rota15, dadeh16, klvo19,  vo19, dash20, sh20, vo20a}. 

To prove Theorem \ref{Linfty thm}, we establish a global Carleman estimate \eqref{Carleman est}. This Carleman estimate is the byproduct of patching together what we call the away-from-origin estimate \eqref{final est} and the near-origin estimate \eqref{obovu est}. 

The away-from-origin estimate is an application of the so called energy method, a well established tool for proving semiclassical Carleman estimates. In particular, we combine and update the approaches from \cite[Section 3]{gash22b} and \cite[Section 3]{ob23}, to construct the weight $w(r)$ and phase $\varphi(r)$, which are key inputs to the energy method. 

Near the origin, $w(r)$ should vanish like $r^2$, to absorb the singular behavior of both $V$ and, in dimension two, the so called effective potential $r^{-2}(n-1)(n-3)$ (the latter arising after we separate variables in Section \ref{prelim section}). In our situation, $V_0$ is only $L^\infty$ near zero. In the proof of Proposition \ref{p:key} below, this necessitates $8 - 4\beta - \beta^2 >0$, which is a stronger requirement than what is needed if $V_0$ has some radial regularity (see \cite[Section 3]{ob23}). This is the source of the discrepancy between \eqref{beta} and \eqref{beta tilde}.

 Away from the origin, roughly speaking, $w(r) > 0$ should increase and have $w'(r) \sim \langle r \rangle^{-s}$, to furnish the weights appearing in \eqref{defn g}. Meanwhile, the main task of $\varphi'(r) > 0$ is control $V_S$ without becoming too large, so as to keep $\varphi(r)$ bounded. Since $V_S$ may decay slowly toward infinity, this is a delicate balancing act, and the compromise we strike is that $\varphi'(r)$ have comparably slow decay for $r > h^{-M}$ and suitable $M \gg 1$, see \eqref{phi prime}. Our choice of $M$, see \eqref{defn M}, is inspired by \cite[Section 2]{vo20b} and more refined compared to \cite[Section 3]{gash22b}. This is why we can handle decay slower than that treated in \cite{gash22b}.

The near-origin estimate was proved by Obovu \cite[Lemma 2.2]{ob23} using the Mellin transform, building on an earlier study of radial potentials \cite{dgs23}. It makes up for the loss in the away-from-origin estimate stemming from the vanishing of $w(r)$ as $r \to 0$. We emphasize that this vanishing of $w(r)$ is essential in dimension two, even if $V_0$ is not singular, because in that case the effective potential has an unfavorable sign.

Resolvent bounds like \eqref{Linfty resolv est delta 1}, \eqref{Linfty resolv est delta between 0 1}, and \eqref{Linfty resolv est delta 0} have application to local energy decay for the wave equation
\begin{equation} \label{wave equation}
\begin{cases}
(\partial_t^2 - c^2(x)\Delta) u(x,t) = 0, & (x,t) \in \left(\mathbb{R}^n \setminus \Omega \right) \times (0, \infty), \, n \ge 2, \\
 u(x,0) = u_0(x),\\
 \partial_t u(x,0) = u_1(x), \\
 u(t,x) = 0, & (x,t) \in  \partial \Omega \times (0,\infty),
\end{cases}
\end{equation}
where $\Omega$ is a compact (possibly empty) obstacle with smooth boundary, and the initial data are compactly supported. A general logarithmic decay rate was first proved by Burq \cite{bu98, bu02} for $c$ smooth. Similar decay was subsequently established for $\Omega = \emptyset$ and $c \in L^\infty(\R^n ; (0, \infty))$ bounded from above and below and identically one outside of a compact set \cite[Theorem 2]{sh18}. See also \cite{be03, cavo04, bo11, mo16, ga19}. Since Theorem \ref{Linfty thm} allows the potential to be singular as $r \to 0$, we expect \cite[Theorem 2]{sh18} extends to $c$ which tends to $0$ at a point. However, for such a $c$, the low frequency character of the solution to \eqref{wave equation} still needs to be accounted for (see, e.g., \cite[Section 4]{sh18}). This question will be taken up elsewhere. 

It's worth mentioning that, in dimension $n \ge 3$, the hypotheses of Theorem \ref{Linfty thm} hold for potentials $V$ which are ``Coulomb-like" near $r = 0$, i.e., obeying $V = O(r^{-1})$ as $r \to 0$. However, the assumption \eqref{V0} does not capture such behavior in dimension two, because in that case $r^{-1}$ is not in $L^2$ near the origin. For Coulomb-like $V$ in dimension two, one can use a quadratic form to show that $P = -h^2 \Delta + V$ is self-adjoint with respect to $ \mathcal{D} \defeq \{u \in H^1(\R^2): Pu \in L^2(\R^2) \}$ \cite[Proposition 1.1]{ch90}. However, it seems difficult to use the method of this paper to prove resolvent estimates for $(P, \mathcal{D})$. This is because our Carleman estimate holds only for functions in $C^\infty_0(\R^2)$. While it is well known that $C^\infty_0(\R^n)$ is dense in $H^2(\R^n)$ for any $n \ge 2$, it is not evident from the standard result on essential self-adjointness for singular potentials \cite[Theorem 2]{si73} that a similar class of smooth functions is dense in $(\mathcal{D}, \| \cdot \|_{\mathcal{D}}),$ where $\|u\|_{\mathcal{D}} \defeq ( \|Pu\|^2_{L^2} + \| u\|^2_{L^2} )^{1/2}$. This is a technical but nevertheless interesting issue that warrants further study.

 \medskip
\noindent{\textsc{Acknowledgements}} It is a pleasure to thank Kiril Datchev and Jeffrey Galkowski for helpful discussions, as well as the anonymous referee for helpful comments and corrections. The author gratefully acknowledges support from ARC DP180100589, NSF DMS 2204322, and from a 2023 Fulbright Future Scholarship funded by the Kinghorn Foundation and hosted by University of Melbourne. The author affirms that there is no conflict of interest.

\section{Preliminary calculations and overview of proof of Theorem \ref{Linfty thm}} \label{prelim section}

In this section, we set the stage for proving Theorem \ref{Linfty thm} by means of the energy method, which has proven to be a dependable tool for establishing resolvent estimates in low regularity (see, e.g., \cite{cavo02, da14, gash22b, ob23}). Throughout this section, we take $P$ as in \eqref{P}, and assume the potential $V$ obeys \eqref{V0} through \eqref{delta}.

 We work in polar coordinates, beginning from the well known identity
\begin{equation*}
    r^{\frac{n-1}{2}}(- \Delta) r^{-\frac{n-1}{2}} = -\partial^2_r + r^{-2} \Lambda,
\end{equation*}
where 
\begin{equation} \label{Lambda}
    \Lambda \defeq -\Delta_{\US^{n-1}} + \frac{(n-1)(n-3)}{4} \ge - \frac{1}{4},
\end{equation}
and $\Delta_{\US^{n-1}}$ denotes the negative Laplace-Beltrami operator on $\US^{n-1}$. Let $\varphi$ be a soon-to-be-constructed phase function on $[0, \infty)$, which is locally absolutely continuous, and obeys $\varphi, \, \varphi' \ge 0$ and $\varphi(0) = 0$. Using $\varphi$, we form the conjugated operator
\begin{equation} \label{conjugation}
\begin{split}
  P^{\pm}_\varphi(h) &\defeq e^{\frac{\varphi}{h}} r^{\frac{n-1}{2}}\left( P(h) - E \pm i\varepsilon \right) r^{-\frac{n-1}{2}} e^{-\frac{\varphi}{h}}\\
  &= -h^2\partial^2_r + 2h \varphi' \partial_r + h^2r^{-2} \Lambda + V -(\varphi')^2 + h\varphi''  - E \pm i\varepsilon.
 \end{split}
\end{equation}

For $u \in e^{\varphi/h} r^{(n-1)/2} C^\infty_0(\R^n)$, define a spherical energy functional,
\begin{equation} \label{F high dim}
    F(r) = F[u](r) \defeq \|hu'(r, \cdot)\|^2 - \langle (h^2r^{-2}\Lambda + V_L  -(\varphi')^2- E)u(r, \cdot), u(r, \cdot) \rangle,
\end{equation}
where $\| \cdot \|$ and $\langle \cdot, \cdot \rangle$ denote the norm and inner product on $L^2(\mathbb{S}_\theta^{n-1})$, respectively. For a weight $w\in C^0[0,\infty)$ that is piecewise $C^1$, the distribution $(wF)'$ on $(0,\infty)$ is given by

\begin{equation} \label{deriv wF}
\begin{split}
    (wF)' &= w'F + wF' \\
    &= w'\|hu'\|^2 - w'\langle (h^2r^{-2} \Lambda + V_L - (\varphi')^2- E)u, u \rangle \\
    & -2 w \real \langle P^{\pm}_\varphi(h) u, u' \rangle + 2wr^{-1} \langle h^2 r^{-2}\Lambda u, u \rangle + w((\varphi')^2 -V_L )' \| u\|^2 + 4h^{-1} w \varphi' \|hu'\|^2 \\
    & \mp 2\varepsilon w \imag \langle u,u'\rangle + 2 w\real \langle (V_0 + V_S+ h \varphi'') u, u' \rangle \\
    &= -2 \real w \langle P^{\pm}_\varphi(h) u, u' \rangle \mp 2\varepsilon w \imag \langle u,u'\rangle + w q \langle h^2 r^{-2} \Lambda u,u\rangle \\
    &+ (4h^{-1}w \varphi' + w')\|hu'\|^2  + (w(E+ (\varphi')^2- V_L ))' \|u\|^2 + 2w\real \langle (V_0 + V_S + h \varphi'') u, u' \rangle.
    \end{split}
\end{equation}
where we have put
\begin{equation} \label{q}
q = q(r) \defeq \frac{2}{r} - \frac{w'}{w}.
\end{equation}

We shall construct $w$ so that $w, \, w' > 0$ and $q \ge 0$. Then using \eqref{deriv wF} and $2ab \ge -(\gamma a^2 + \gamma^{-1}b^2)$ for all $\gamma > 0$, we find
\begin{equation}
\label{e:lowerDerivative1}
\begin{aligned}
    w' F + w F' &\ge -\frac{\gamma_1 w^2}{h^2w'} \| P^{\pm}_\varphi(h)u \|^2 \mp  2\varepsilon w \imag \langle u,u'\rangle  \\
    &+ (4( 1 -\gamma^{-1}_2)h^{-1} w \varphi' + (1 - \gamma^{-1}_1 - \gamma^{-1}_2) w')\|hu' \|^2 \\
    &+ \big((w(E + (\varphi')^2 - V_L))' - \frac{h^2 wq}{4r^2}   -\frac{\gamma_2 w^2 |h^{-1}(V_0 + V_S) + \varphi''|^2}{w' + 4h^{-1}\varphi' w} \big)\|u\|^2 , \qquad \gamma_1, \, \gamma_2 > 0.
    \end{aligned}
\end{equation}
For $\gamma > 0$, put $\gamma_1 = 2(1 + \gamma)/\gamma, \, \gamma_2 = (1 + \gamma)$, yielding
\begin{equation}
\label{e:lowerDerivative2}
\begin{split}
    (wF)'
      &\ge -\frac{2(1 + \gamma)w^2}{\gamma h^2w'} \| P^{\pm}_\varphi(h) u \|^2 \mp 2 \varepsilon w \imag \langle u, u' \rangle +  \frac{\gamma}{2(1 + \gamma)} w' \|hu'\|^2  \\
    &+ \big((w(E + (\varphi')^2 - V_L))' - \frac{h^2 wq}{4r^2}   -\frac{(1 + \gamma)w^2 |h^{-1}(V_0 + V_S) + \varphi''|^2}{w' + 4h^{-1}\varphi' w} \big)\|u\|^2. 
    \end{split}
\end{equation}

In Section \ref{resolv est section}, we show how Theorem \ref{Linfty thm} follows from a certain global Carleman estimate, see Lemma \ref{Carleman lemma}. An essential ingredient for this Carleman estimate is to specify $\varphi$ and $w$ as precisely as possible, in order that the second line of \eqref{e:lowerDerivative2} has a good lower bound. More precisely, putting 
\begin{equation} \label{A and B}
A(r)\defeq (w(E+ (\varphi')^2 - V_L))' - \frac{h^2 wq}{4r^2},\qquad B(r) \defeq \frac{w^2|h^{-1}(V_0 +V_S) + \varphi''|^2}{w' + 4h^{-1}\varphi' w},
\end{equation}
we shall see that it suffices for $w$ and $\varphi$ to satisfy, for suitable $\gamma > 0$,
\begin{equation}
\label{e:goalEst}
A(r)- (1 + \gamma)B(r) \geq \frac{E}{2}w'(r), \qquad 0 < h \ll 1.
\end{equation}
To facilitate the proof of \eqref{e:goalEst}, we proceed, as in \cite{gash22a, gash22b, ob23}, to analyze $A$ and $B$ in terms of the auxiliary functions 
\begin{equation}\label{e:defPhiW}
\Phi \defeq \frac{\varphi''}{\varphi'}=(\log |\varphi'|)',\qquad \mathcal{W} \defeq \frac{w}{w'}=\frac{1}{(\log |w|)'}.
\end{equation}
In particular, from \eqref{A and B} and \eqref{e:defPhiW},

\begin{equation} \label{e:keyCalc}
\begin{split}
A(r)-(1 + \gamma)B(r)&\geq w'\Big[ E+(\varphi')^2\big(1+2\mathcal{W}\Phi - 2(1 + \gamma) \mathcal{W}|\Phi|^2 \min \big( \mathcal{W}, \frac{h}{4 \varphi'} \big)\big)  \\
&-2(1 + \gamma)h^{-2}\mathcal{W} |V_0 + V_S|^2 \min \big( \mathcal{W}, \frac{h}{4 \varphi'} \big)-V_L-\mathcal{W} \big(V'_L  + \frac{h^2q}{4r^2}\big)\Big].
\end{split}
\end{equation}

So, to show \eqref{e:goalEst}, it is enough to bound the bracketed expression in \eqref{e:keyCalc} from below by $E/2$. The next section is devoted to constructing $w$ and $\varphi$, and their corresponding $\mathcal{W}$ and $\Phi$, that will bring about \eqref{e:goalEst}. 

\section{Determination of the weight and phase} \label{weight and phase section}

In this section, we develop the functions $w$ and $\varphi$, and their associated $\mathcal{W}$ and $\Phi$, as in \eqref{e:defPhiW}. They play an essential role in the proof of the lower bound \eqref{e:goalEst} for $A - (1 + \gamma)B$ (Proposition \ref{p:key}), and in the proof of the Carleman estimate (Lemma \ref{Carleman lemma}). We should keep in mind that $A$ and $B$ (see \eqref{A and B}) depend not only on $w$ and $\varphi$, but also on a potential $V$ that obeys \eqref{V0} to \eqref{delta}. 

First, we fix
\begin{gather}
\sigma \defeq \frac{1}{3}, \label{sigma} \\
\tilde{\rho} \in (1, \rho], \qquad \text{$\rho$ as in \eqref{delta}} \label{tilde rho}.
\end{gather}
Using \eqref{VL r big} and \eqref{VL prime r big}, fix $b > 0$ independent of $h$ large enough so that 
\begin{equation} \label{b}
V_L, \, \frac{r}{2}V_L' \le \frac{E_{\min}}{8}, \qquad r > b. 
\end{equation}

Next, we introduce several quantities depending on the semiclassical parameter $h$ and on $\delta$ as in \eqref{VS bound}. These quantities also involve parameters  $T > 0$,  $t \ge 1$ that are independent of $h$ and will be specified in the proof of Proposition \ref{p:key}:
\begin{gather}
 \lambda \defeq  \log (\log(h^{-1})), \label{lambda}\\
\eta \defeq (\log(h^{-1}))^{-1}, \label{eta} \\
k \defeq \begin{cases} 1 & \delta = 1,
 \\ \frac{1 + 2\delta - \lambda^{-1}}{3} & 0 < \delta < 1, \\
\frac{1}{3} & \delta = 0, \end{cases} \label{k} \\
  M \defeq \begin{cases}  \frac{\sigma}{k} + T \eta \lambda  & 0< \delta \le 1, \\
   1 + T \eta \lambda + t \eta & \delta = 0, \end{cases} \label{defn M}\\
a \defeq h^{-M}. \label{defn a} 
\end{gather}
In this and later sections, we always assume $h$ is restricted to $(0, h_\delta]$, where $h_\delta \in (0,1)$ is small enough so that
\begin{equation} \label{implication}
h \in (0, h_\delta] \implies \eta \lambda \in (0, 1], \, k \in [\tfrac{1}{3}, 1], \text{ and }\eta \le \begin{cases} \min(\delta, \frac{1}{3}) & 0 < \delta \le 1 \\
1 & \delta = 0
 \end{cases}.
\end{equation}
In particular, from $h^{\eta \lambda} = \eta$, and $h^{\eta} = e^{-1}$, \eqref{defn M}, \eqref{defn a}, \eqref{implication},
 \begin{equation} \label{lwr bd a}
 a = h^{-M} = \begin{cases} 
  h^{-\frac{\sigma}{k} - T\eta \lambda} =  h^{-\frac{\sigma}{k}} (\log(h^{-1}))^T \ge h^{-\frac{1}{3}} & 0 < \delta \le 1 \\
   h^{-1- T\eta \lambda - t\eta} = e^{t} h^{-1} (\log(h^{-1}))^T \ge h^{-1} & \delta = 0
  \end{cases}.
 \end{equation}

Our weight $w$ and phase $\varphi$ are:
\begin{gather}
 \label{w}
w(r) \defeq \begin{cases} r^2 &0 < r \le a, \\
a^2e^{\int_a^r \max ( \frac{2}{s \mathcal{G}(s)}, \frac{4m_L(s)}{Es}) ds} & r > a.
\end{cases}\\
\varphi_0'(r) \defeq \begin{cases}
 r^{-\tfrac{\beta}{2}} & 0 < r \le 1,\\[6pt]
 e^{-\int_{1}^r \frac{k}{s+\Phi_1(s)}ds}& 1 < r \le a,\\[6pt]
\varphi_0'(a)\frac{a\mathcal{G}(a)}{r\mathcal{G}(r)}&r > a,
\end{cases} \label{phi prime}\\
\varphi_0(r) \defeq \int^r_0\varphi'_0(s)ds, \qquad r > 0, \label{phi0} \\
\varphi(r) \defeq \tau h^{-\sigma} \varphi_0(r), \qquad r > 0, \, \tau \ge 1, \label{tau}
\end{gather}
where
\begin{gather}
\mathcal{G}(r) \defeq \begin{cases} r^\eta & \delta > 0, \\
(\log r)^{\tilde{\rho}} & \delta = 0,
 \end{cases} \label{G}
 \\  \Phi_1 \defeq \frac{\kappa r(\tilde{m}^2_S + \chi + (y + m_L) \mathbf{1}_{1 < r \le b})}{1 - \kappa(\tilde{m}^2_S + \chi + (y + m_L) \mathbf{1}_{1 < r \le b})},   \qquad \kappa \in (0, 1/8], \label{Phi1} \\
 \tilde{m}_S(r) = \begin{cases} m_S(r) & \delta = 0 \text{ or }1,  \\ 
 r^{- \lambda^{-1}/2} & 0 < \delta < 1.
 \end{cases} \label{tildem} 
\end{gather}
The parameters $\tau$ and $\kappa$ are independent of $h$ and will be fixed in the proof of Proposition \ref{p:key}. Note that the denominator of $\Phi_1$ is at least $1/2$ since $0 \le \tilde{m}_S, \chi, y, m_L \le 1$ and $\kappa \in (0, 1/8]$, where $y$ is given by \eqref{y}.

Recalling that $\mathcal{W}$ and $\Phi$ are defined by \eqref{e:defPhiW}, we use \eqref{w} and \eqref{phi prime} to calculate
\begin{gather} \
\mathcal{W}(r) = \begin{cases} 
\frac{r}{2} & 0 < r < a, \\[6pt]
\frac{r}{2} \min( \mathcal{G}(r), \frac{E}{2m_L(r)} ) & r > a,
\end{cases} \label{W} \\
\Phi(r) = \begin{cases} 
-\frac{\beta}{2r} & 0 < r < 1, \\[6pt]
 - \frac{k}{r  + \Phi_1(r)}  & 1 < r < a, \\[6pt]
-\frac{1 + \eta}{r} & 0 < \delta \le 1,  \, r > a, \\[6pt]
-\frac{1 + \tilde{\rho}(\log r)^{-1} }{r} & \delta = 0, \, r > a. 
\end{cases} \label{Phi} 
\end{gather}

To conclude this section, we collect several basic properties of $w$, and an elementary Lemma about $\Phi_1$, which are important to the proofs of the lower bound \eqref{e:goalEst} and the Carleman estimate.

\begin{lemma}
There exists $C$ independent of $h$ so that for all $h \in (0,h_\delta]$, 
\begin{gather}
w(r) \le C h^{-2-2M} \label{univ bd w}, \qquad r > 0,\\
    w'(r) \ge (\log(h^{-1}))^{-C} r^{-1-\eta} \label{univ lwr bd w prime}, \qquad r > a,  \\
    \frac{w(r)^2}{w'(r)} \le C (\log(h^{-1}))^{C} h^{-2 - 2M} r^{1+\eta}, \qquad \, r \neq  a. \label{w squared over w prime}
 \end{gather}
\end{lemma}

\begin{proof}
 To see~\eqref{univ bd w}, note that from \eqref{w} $w$ is clearly increasing, so we need only compute $\limsup_{r\to \infty} w(r)$. By \eqref{eta}, \eqref{defn a} and \eqref{G}, for some $C > 0$ independent of $h \in (0,h_\delta]$,
 \begin{equation*}
\begin{split}
\limsup_{r\to \infty}w(r) &=\limsup_{r \to \infty} \, a^2e^{\int_a^r \max ( \frac{2}{s \mathcal{G}(s)}, \frac{4 c_L m_L(s)}{Es}) ds}\\
& \le \begin{cases} Ca^2 e^{2/\eta}= C h^{-2 -2M}  & \delta > 0\\ 
Ca^2 = Ch^{-2M} & \delta = 0 \end{cases}. 
\end{split}
\end{equation*}

For~\eqref{univ lwr bd w prime}, we use \eqref{w} to compute $w'$ for $r > a$:
$$
w'(r) = a^2 e^{\int_a^r \max ( \frac{2}{s \mathcal{G}(s)}, \frac{4m_L(s)}{Es}) ds} \max \big( \frac{2}{r \mathcal{G}(r)}, \frac{4m_L(r)}{Er}\big)\geq (\log(h^{-1}))^{-C}r^{-1-\eta},
$$
for some constant $C > 0$ independent of $h \in (0,h_\delta]$, where when $\delta = 0$ we have used that
\begin{equation*}
\frac{r^\eta}{(\log r)^{\tilde{\rho}}} = \Big( \frac{r^{\eta/\tilde{\rho}}}{\log r}  \Big)^{\tilde{\rho}} \ge \Big( \frac{\log(r^{\eta/\tilde{\rho}})}{\log r}  \Big)^{\tilde{\rho}} =  \frac{\eta^{\tilde{\rho}}}{\tilde{\rho}^{\tilde{\rho}} }, \qquad r > 1.
\end{equation*}
Finally, \eqref{w squared over w prime} follows from \eqref{univ bd w} and
\begin{equation*}
\frac{w}{w'} = \mathcal{W} \le \begin{cases} \tfrac{r}{2} & 0 < r < a \\[6pt]
\tfrac{r^{1 + \eta}}{2} & r > a, \, 0 < \delta \le 1 \\[6pt]
\tfrac{r(\log r)^{\tilde{\rho}}}{2} & r > a, \, \delta = 0
 \end{cases}.
\end{equation*}
\end{proof}

\begin{lemma} [{\cite[Lemma 2.1]{gash22b}}]
\label{l:phi}
It holds that 
\begin{equation} \label{int Phi bounds} 
-\log r \leq -\int_1^r \frac{1}{s+\Phi_1(s)} ds\leq  -\log r +\|s^{-2}\Phi_1(s)\|_{L^1(1,\infty)}.
\end{equation}
\end{lemma}

\begin{proof}
We note first that $\|s^{-2}\Phi_1(s)\|_{L^1(1,\infty)} < \infty$ thanks to \eqref{Phi1}, \eqref{tildem}, and \eqref{delta}. The estimate \eqref{int Phi bounds} follows by combining
\begin{align*}
\log (r)-\int_1^r \frac{1}{s+\Phi_1(s)}ds&=\int_1^r \frac{1}{s}-\frac{1}{s+\Phi_1(s)}ds\\
&=\int_1^r\frac{\Phi_1(s)}{s(s+\Phi_1(s))}ds
\end{align*}
with 
$$
0\leq \int_1^r\frac{\Phi_1(s)}{s(s+\Phi_1(s))}ds\leq \|s^{-2}\Phi_1(s)\|_{L^1(1,\infty)}.
$$
\end{proof}

\section{Proof of the main estimate} \label{main estimate section}

\begin{proposition}  
\label{p:key}
Suppose $V$ satisfies \eqref{V0} through \eqref{delta}. Fix $[E_{\min}, E_{\max}] \subseteq (0, \infty)$, $\e \in (0, 1)$, $\tilde{\rho} \in (1, \rho]$, and
\begin{equation} \label{gamma}
0 < \gamma < (8 - 4\beta - \beta^2)/\beta^2.
\end{equation}
Let $w$ and $\varphi$ be as constructed in Section \ref{weight and phase section}. 

There exist $T >0$ and $t \ge 1$ as in \eqref{defn M}, $\tau \ge 1$ as in \eqref{tau}, $\kappa \in (0,1/8]$ as in \eqref{Phi1},   $C_\delta > 0$, and $h_\delta \in (0,1)$, all independent $h$, so that
\begin{equation} \label{e:keyLower}
A(r)-(1 + \gamma)B(r)\geq \frac{E_{\min}}{2}w'(r), \qquad  E \in [E_{\min}, E_{\max}],\, h \in (0,h_\delta], \,  r \neq 1, \, a,
\end{equation}
and
\begin{equation} \label{e:phiBoundMe}
|\varphi_0(r)| 
\le  \begin{cases} C_\delta \log(h^{-1}) & \delta = 1 \\   C_\delta h^{-\frac{2(1 - \delta) + \lambda^{-1} }{3(1 + 2\delta - \lambda^{-1})}}(\log(h^{-1}))^{1 +  \e} & 0 < \delta <1 \\
C_\delta h^{-\frac{2}{3}}(\log(h^{-1}))^{1 + \tilde{\rho}} & \delta = 0 \end{cases}, \qquad h \in (0, h_\delta], \, r > 0. 
 \end{equation}
\end{proposition}

\begin{proof} 
We prove Proposition \ref{p:key} over the course of subsections \ref{small r section}, \ref{large r section}, and \ref{bd phase section}. Throughout the proof, $C$ denotes a positive constant whose value may change from line to line, but is always independent of $T, \, t, \, \tau, \, \kappa,$ and $h$. Initially, we take $h_\delta \in (0,1)$ small enough so that \eqref{implication} holds. At several steps of the proof, we further decrease $h_\delta$ if necessary.

\subsection{Proof of \eqref{e:keyLower}, small $r$ region} \label{small r section} \hfill \\
\textbf{Case $0 < r < 1$:}\\
When $0 < r< 1$,
\begin{equation*}
\mathcal{W} = \frac{r}{2}, \quad q = 0 \text{ (see \eqref{q})}, \quad \Phi = -\frac{\beta}{2r}, \quad \varphi' = \tau h^{-\sigma} r^{-\beta/2}.
\end{equation*}

  Using these, $\sigma = 1/3$, and that $|V_0| \le c_0r^{-\beta}$ and $V_L = V_S = 0$ in a neighborhood of $r \le 1$ (see \eqref{V0 bd} and \eqref{L plus S}), we revisit \eqref{e:keyCalc} and find
 \begin{equation*}
 \begin{gathered}
 1 + 2 \mathcal{W} \Phi = 1 - \frac{\beta}{2} ,\\
 2(1 + \gamma) \mathcal{W} |\Phi|^2 \min \big(\mathcal{W}, \frac{h}{4\varphi'} \big) \le  2(1 + \gamma) \mathcal{W}^2 |\Phi|^2  \le \frac{(1 + \gamma) \beta^2}{8} , \\
 2(1 + \gamma) h^{-2} \mathcal{W} |V_0 + V_S|^2 \min \big(\mathcal{W}, \frac{h}{4\varphi'} \big) \le (1 + \gamma)\frac{c_0^2}{4\tau} h^{-2\sigma} r^{1 - \frac{3\beta}{2}}, \\ 
 V_L + \mathcal{W}\big(V'_L + \frac{h^2q}{4r^2} \big) = 0.
 \end{gathered}
 \end{equation*}
 In the second estimate, we used that the minimum is less than $\mathcal{W}$, but in the third estimate, we used that it is less than $h/(4 \varphi')$.
 Therefore
\begin{equation*}
\begin{gathered}
   A-(1+\gamma)B \ge w' \big[ E + h^{-2\sigma} r^{-\beta} \big( \tau^2( 1 - \tfrac{\beta}{2} - (1 + \gamma)\tfrac{\beta^2}{8}) -(1 + \gamma) \tfrac{c^2_0}{4 \tau} r^{1 - \frac{\beta}{2}} \big) \big], \\ 
   E \in [E_{\min}, E_{\max}], \, h \in (0,h_\delta], \, 0 < r < 1.
    \end{gathered}
\end{equation*}	
Since $\beta \le 2$ (see \eqref{beta}), and since \eqref{gamma} implies $8 - 4 \beta- (1 + \gamma)\beta^2 > 0$, we may choose $\tau \ge 1$ large enough, independent of $h$, so that \eqref{e:keyLower} holds for $0 < r < 1$. 

\noindent \textbf{Case $1 < r < a$:}\\
When $1 < r < a$,
\begin{equation} \label{quantities med r}
 \mathcal{W} = \frac{r}{2}, \quad  q = 0, \quad \Phi = -\frac{k}{r + \Phi_1(r)}, \quad \varphi' = \tau h^{-\sigma} e^{ -\int_1^r \frac{k}{s + \Phi_1(s)} ds}. 
\end{equation}
We first derive some bounds on $\varphi'_0 = \exp( -\int_1^r \frac{k}{s + \Phi_1(s)} ds)$ (see \eqref{phi prime}), and for this we use Lemma \ref{l:phi}. By \eqref{int Phi bounds},
\begin{equation*}
\frac{1}{r^k}\leq \varphi_0'(r) \leq \frac{e^{k\| s^{-2}\Phi_1(s)\|_{L^1(1, \infty)}}}{r^k}, \qquad 1 \le r \le a.
\end{equation*}

Next, we bound the exponent $k\| s^{-2}\Phi_1(s)\|_{L^1(1, \infty)}$ depending on the value of $\delta$. If $\delta = 0$ or $1$, then both $k$ and $\Phi_1(s)$ are independent of $h$ (see \eqref{k} and \eqref{Phi1}, respectively), thus we simply have $k\| s^{-2}\Phi_1(s)\|_{L^1(1, \infty)} \le C$. On the other hand, when $0 < \delta < 1$, both $k$ and  $\Phi_1(s)$ depend on $h$. But in this case $1/3 \le k \le 1$ thanks to \eqref{implication}, and, by \eqref{Phi1}, \eqref{tildem} and $\lambda = \log(\log(h^{-1}))$,
\begin{equation*}
 \| s^{-2}\Phi_1(s)\|_{L^1(1, \infty)} \le C(1 + \kappa \int^\infty_1 s^{-1 -\lambda^{-1}}ds) \le C(1 + \kappa \log(\log(h^{-1}))).
\end{equation*}
Thus we conclude
\begin{equation}
\label{e:phiBounds}
\begin{split}
\frac{1}{r^k}\leq \varphi_0'(r)&\leq \frac{e^{k\| s^{-2}\Phi_1(s)\|_{L^1(1, \infty)}}}{r^k}\\
&\le \begin{cases}   \frac{C}{r^{k}}   & \delta = 0 \text{ or } 1 \\[6pt]
\frac{C (\log(h^{-1}))^{\kappa C}}{r^k}  &0 < \delta < 1
\end{cases} , \qquad h \in (0, h_\delta], \, 1 \le r \le a.
\end{split}
\end{equation}
The estimate \eqref{e:phiBounds} informs our choice of $\kappa$. If $\delta = 0$ or $1$, fix $\kappa = 1/8$. If $0 < \delta < 1$, fix $\kappa \in (0, 1/8]$ small enough so that the factor $(\log(h^{-1}))^{\kappa C}$ in \eqref{e:phiBounds} is bounded from above by $(\log(h^{-1}))^{\e_1}$, where 
\begin{equation} \label{ep}
\e_1 \defeq \frac{\e}{7}.
\end{equation}

 So with $\kappa$ now fixed, we have
\begin{equation}
\label{e:phiBounds2}
\frac{1}{r^k}\leq \varphi_0'(r) \le \begin{cases}   \frac{C}{r^{k}}   & \delta = 0 \text{ or } 1 \\[6pt]
\frac{C (\log(h^{-1}))^{\e_1}}{r^k}  &0 < \delta < 1
\end{cases} , \qquad h \in (0, h_\delta], \, 1 \le r \le a.
\end{equation}

As in the previous case, we estimate each of the terms on the right side of \eqref{e:keyCalc}, keeping in mind that now $1 < r < a$. By $\mathcal{W} = r/2$, \eqref{Phi}, \eqref{Phi1}, the lower bound in \eqref{e:phiBounds2}, and $\varphi' = \tau h^{-\sigma} \varphi'_0$,
\begin{equation*}
 \begin{gathered}
 (\varphi')^2(1 + 2 \mathcal{W} \Phi) = (\varphi')^2 \big(1 - \frac{kr}{r + \Phi_1} \big) = (\varphi')^2\frac{(1 - k)r + \Phi_1}{r + \Phi_1} \\
 \ge \frac{\tau^2 h^{-2\sigma}}{r^{2k}} \frac{\Phi_1}{r + \Phi_1} = \frac{\tau^2 \kappa h^{-2\sigma}}{r^{2k}} (\tilde{m}^2_S + \chi + (y + m_L) \mathbf{1}_{1 < r \le b}).
 \end{gathered}
\end{equation*}
Continuing on, we use $\mathcal{W} = r/2$, \eqref{Phi}, the upper bound in \eqref{e:phiBounds2}, $\min( \mathcal{W}, h/(4\varphi')) \le h/(4\varphi')$, $\varphi' = \tau h^{-\sigma} \varphi'_0$, and $k \le 1$ to find
 \begin{equation*}
 2(1 + \gamma)(\varphi')^2 \mathcal{W} |\Phi|^2 \min \big(\mathcal{W}, \frac{h}{4\varphi'} \big) \le (1 + \gamma)(\varphi')^2 \frac{k^2r}{(r + \Phi_1)^2} \min \big(\mathcal{W}, \frac{h}{4\varphi'} \big) \le  C \tau h^{1 - \sigma} \log(h^{-1})^{\e_1}.
 \end{equation*}
 To estimate the next term, we use $\mathcal{W} = r/2$, $|V_S + V_0|^2 \le C(m^2_S + \chi)r^{-2-2\delta}$ for $1 < r < a$, $\varphi' = \tau h^{-\sigma} \varphi'_0$, $\tau \ge 1$, and
 \begin{equation*}
  \min( \mathcal{W}, h/(4\varphi')) \le  \frac{h}{4\varphi'} \le \frac{h^{1+\sigma}} {\tau \varphi'_0} \le h^{1+\sigma} r^k, \qquad 1 < r < a, 
 \end{equation*}
 to see
 \begin{equation*}
 \begin{gathered}
 2(1 + \gamma) h^{-2} \mathcal{W} |V_S + V_0|^2 \min \big(\mathcal{W}, \frac{h}{4\varphi'} \big) \le \frac{Ch^{-1 + \sigma}(m^2_S + \chi)}{r^{1 + 2\delta - k}} \\
 = \begin{cases} 
 \frac{Ch^{-1 + \sigma}(m^2_S + \chi)}{r^{2k}} =  \frac{Ch^{-1 + \sigma}(\tilde{m}^2_S + \chi)}{r^{2k}} & \delta = 0 \text{ or } 1 \\[6pt]
 \frac{Ch^{-1 + \sigma}(m^2_S + \chi)}{r^{2k + \lambda^{-1}}} \le  \frac{Ch^{-1 + \sigma}(\tilde{m}^2_S + \chi)}{r^{2k}}  & 0< \delta < 1
 \end{cases}.
 \end{gathered}
 \end{equation*}
 To finish the estimate we used \eqref{delta}, \eqref{k} and \eqref{tildem}. In particular, when  $\delta = 0$ or $1$, $1 + 2\delta - k = 2k$ and $m_S = \tilde{m}_S$; when $0 < \delta < 1$, $1 + 2\delta - k = 2k + \lambda^{-1}$, $m_S = 1$, and $\tilde{m}_S = r^{-\lambda^{-1}/2}$.
 
 To estimate the final term we use $\mathcal{W} = r/2$, $q = 0$, and \eqref{b}, yielding
 \begin{equation*}
 V_L + \mathcal{W}\big(V'_L + \frac{h^2q}{4r^2} \big) = V_L + \frac{r}{2} V'_L   \le \big(c_L y  + \frac{c_L}{2} m_L \big)\mathbf{1}_{1 < r \le b} + \frac{E_{\min}}{4}.
 \end{equation*}
 Putting the above bounds into \eqref{e:keyCalc} and recalling $\sigma = 1/3$ yields
\begin{equation} \label{lwr bd med r 2}
\begin{gathered}
A-(1 + \gamma) B \ge w'\Big[\frac{3E}{4}+\frac{h^{-2\sigma}}{r^{2k}} (\tau^2 \kappa -C) (\tilde{m}^2_S + \chi + (y + m_L) \mathbf{1}_{1 < r \le b}) - C \tau h^{1 -\sigma} (\log(h^{-1}))^{\e_1} \Big], \\
E \in [E_{\min}, E_{\max}], \, h \in (0, h_\delta], \, 1 < r < a.
\end{gathered}
\end{equation}
Modify $\tau$ to be the maximum of $(C/\kappa)^{1/2}$ and its value assigned previously. Restricting $h$ further so that $C \tau h^{1 -\sigma} (\log(h^{-1}))^{\e_1} \le E_{\min}/4$, we arrive at \eqref{e:keyLower} when $1 < r < a$. 

\subsection{Proof of \eqref{e:keyLower}, large $r$ region} \label{large r section} \hfill \\
When $r > a$,
\begin{equation} \label{quantities large r}
\begin{gathered}
 \mathcal{W} = \frac{r}{2} \min \big( \mathcal{G}(r), \frac{E}{2c_Lm_L(r)}\big), \quad  q = \big( \frac{2}{r} - \frac{1}{\mathcal{W}} \big), \\
  \Phi(r) \defeq \begin{cases} -\frac{1 + \eta}{r}  & \delta > 0 \\
-\frac{1 + \tilde{\rho} (\log r)^{-1}}{r}  & \delta = 0
 \end{cases}, \quad \varphi' = h^{-\sigma} \tau \varphi'_0(a) \frac{a \mathcal{G}(a)}{r \mathcal{G}(r)}. 
 \end{gathered}
 \end{equation}
Combining the two identities in the first line of \eqref{quantities large r}, and substituting the expression for $\mathcal{G}$ (see \eqref{G}), shows
\begin{equation*}
\frac{r}{2}q = 1 - \max \big( \frac{1}{\mathcal{G}(r)}, \frac{2c_Lm_L(r)}{E}\big) = \begin{cases}
1 - \max \big( r^{-\eta}, \frac{2c_Lm_L(r)}{E}\big) & 0 < \delta < 1 \\
1 - \max \big( (\log r)^{-\tilde{\rho}}, \frac{2c_Lm_L(r)}{E}\big) & \delta = 0.
\end{cases}.
\end{equation*}
Recall from \eqref{e:lowerDerivative1} that we need to have $q \ge 0$. Since we have previously arranged $a \ge h^{-1/3}$ for all $0 \le \delta \le 1$ (see \eqref{lwr bd a}), and because $\lim_{r \to \infty} m_L(r) = 0$ (see \eqref{mL}), we may take $h_\delta$ smaller, if needed, so $h \in (0, h_\delta]$ and $r > a$ imply $q \ge 0$.

From the second line of \eqref{quantities large r},
 \begin{equation} \label{bounds Phi large r}
  -\frac{2}{r} \le \Phi(r) \le -\frac{1}{r}, \qquad r > a.
  \end{equation}
From $\varphi'_0(r) = \varphi'_0(a) a \mathcal{G}(a)/(r \mathcal{G}(r))$ for $r \ge a$ and \eqref{e:phiBounds2}, we also find
\begin{equation} \label{est phi0 prime r large}
\begin{split}
\frac{a^{1 - k }\mathcal{G}(a)}{r\mathcal{G}(r)} &\leq \varphi_0'(r) \\
&\le \begin{cases}
 \frac{Ca^{1-k}\mathcal{G}(a)}{r\mathcal{G}(r)} \le Ca^{-k} & \delta = 0 \text{ or } 1  \\
  \frac{C(\log(h^{-1}))^{\e_1}a^{1-k}\mathcal{G}(a)}{r\mathcal{G}(r)} \le C(\log(h^{-1}))^{\e_1} a^{-k} & 0 < \delta < 1
\end{cases} ,\quad h \in (0, h_\delta], \, r \ge a.
\end{split}
\end{equation}

We now make additional calculations involving $a = h^{-M}$ that are crucial below. This is where we make use of the parameters $T > 0$, $t \ge 1$ that were introduced in the definition of $M$, see \eqref{defn M}.  First, consider when $0 < \delta \le 1$. By our standing assumption \eqref{implication}, we have $1/3 \le k \le 1$ and $2k - \eta \ge 1/3$. Using the definition of $k$ (see \eqref{k}), it is straightforward to verify that $1 + 2\delta - k \ge 2k$ too. Thus, from $h^{\eta \lambda} = \eta$, $h^{\eta} = e^{-1}$, $M = (\sigma/k) + T\eta \lambda$ (see \eqref{defn M}), and $a = h^{-\frac{\sigma}{k}} (\log(h^{-1}))^T \ge h^{-\frac{1}{3}}$ (see \eqref{lwr bd a}),
  \begin{equation} \label{unpack a delta posit}
 \begin{gathered} 
 h^{-2\sigma} a^{-2k + \eta} = h^{-2\sigma}  (h^{\frac{\sigma}{k}} (\log(h^{-1}))^{-T})^{2k - \eta} \le e^{\frac{\sigma}{k}} \log(h^{-1})^{-\frac{T}{3}}  \le C \log(h^{-1})^{-\frac{T}{3}}, \\
 h^{-2\sigma} a^{-1 - 2\delta + k + \eta}  \le h^{-2\sigma} a^{-2k + \eta} \le C \log(h^{-1})^{-\frac{T}{3}}.
  \end{gathered}
 \end{equation}
We will fix the parameter $T$ for this case later.

  Second, consider when $\delta = 0$. Then $k = 1/3$ and we fix $T = \tilde{\rho}/2k = 3\tilde{\rho}/2$ at once. Furthermore, $M = 1 + (3\tilde{\rho} \eta \lambda /2) + t \eta$ (see \eqref{defn M}) and $a = e^{t} h^{-1} (\log(h^{-1})^{\frac{3\tilde{\rho}}{2}} \ge h^{-1}$ (see \eqref{lwr bd a}). Then for $h_\delta$ small enough and $h \in (0, h_\delta]$, $ 1 \le M \le Ct$. Using also $\sigma = 1/3$,
 \begin{equation} \label{unpack a delta 0}
 \begin{gathered}
h^{-2\sigma} a^{- 2k} (\log a)^{\tilde{\rho}} = h^{-\frac{2}{3}} a^{-\frac{2}{3}} (\log a)^{\tilde{\rho}} = h^{-\frac{2}{3}}  (e^{-t}h (\log(h^{-1}))^{-\frac{3\tilde{\rho}}{2}})^{\frac{2}{3}} (\log a)^{\tilde{\rho}} \\
 = M^{\tilde{\rho}} e^{-\frac{2t}{3}} \le  C t^{\tilde{\rho}} e^{-\frac{2t}{3}}, \\
 h^{-2\sigma} a^{-1 + k}  (\log a)^{-\tilde{\rho}}  = h^{-\frac{2}{3}} (e^{-t}h (\log(h^{-1}))^{-\frac{3\tilde{\rho}}{2}})^{\frac{2}{3}}  (\log a)^{-\tilde{\rho}}  \le e^{-\frac{2t}{3}}.
  \end{gathered}
 \end{equation}

Once more, our goal is to control the terms on the right side of \eqref{e:keyCalc}, but this time for $r > a$. First, by $\mathcal{W} \le r \mathcal{G}/2$ (see \eqref{quantities large r}), $|\Phi| \le 2/r$ (see \eqref{bounds Phi large r}, $\varphi' = \tau h^{-\sigma} \varphi'_0$, and the upper bound in \eqref{est phi0 prime r large},

\begin{equation*}
2 (\varphi')^2 \mathcal{W}|\Phi| \le 2 (\varphi')^2 \mathcal{G} \le \begin{cases} C\tau^2 h^{-2\sigma} a^{-2k + \eta} & \delta = 1 \\
C\tau^2 h^{-2\sigma} (\log(h^{-1}))^{2\e_1} a^{-2k + \eta}  & 0 < \delta < 1 \\ 
C\tau^2 h^{-2\sigma} a^{-2k} (\log a)^{\tilde{\rho}} & \delta = 0
 \end{cases}.
\end{equation*}
Note that in the last step we used $\mathcal{G} = r^{\eta}$ for $0 < \delta \le 1$ while $\mathcal{G} = (\log r)^{\tilde{\rho}}$ for $\delta = 0$. Using now the first line of \eqref{unpack a delta posit} and the first line of \eqref{unpack a delta 0},
\begin{equation*}
2 (\varphi')^2 \mathcal{W}|\Phi| \le \begin{cases} C \tau^2 \log(h^{-1})^{-\frac{T}{3}} & \delta = 1 \\
C \tau^2 \log(h^{-1})^{2\e_1-\frac{T}{3}}  & 0 < \delta < 1 \\ 
 C \tau^2t^{\tilde{\rho}} e^{-\frac{2t}{3}} & \delta = 0
 \end{cases}.
\end{equation*}

Next, by $\mathcal{W} \le r \mathcal{G}/2$, $|\Phi| \le 2/r$, $\min(\mathcal{W}, h/(4 \varphi')) \le h/(4 \varphi')$, $\varphi' = \tau h^{-\sigma}\varphi'_0$, the upper bound in \eqref{est phi0 prime r large},

 \begin{equation*}
 2(1 + \gamma) (\varphi')^2 \mathcal{W} |\Phi|^2 \min \big(\mathcal{W}, \frac{h}{4\varphi'} \big) \le \frac{C \tau h^{1 - \sigma} \mathcal \varphi'_0 \mathcal{G}}{r} \\
  \le \begin{cases} C\tau h^{1 - \sigma} a^{-1-k + \eta}  & \delta = 1\\[6pt] 
 C\tau (\log(h^{-1}))^{\e_1} h^{1 - \sigma} a^{-1-k + \eta} & 0 < \delta < 1 \\[6pt]
 C\tau h^{1 - \sigma} a^{-1-k} (\log a)^{\tilde{\rho}} & \delta = 0 
   \end{cases}.
\end{equation*}
 Since $a \ge h^{-1/3}$ for all $0 \le \delta \le 1$ (see \eqref{lwr bd a}), by further decreasing $h_\delta$ as needed we attain $(\log(h^{-1}))^{\e_1} a^{-1 -k  +\eta} \le 1$ for $0 < \delta \le 1$ and $h \in (0,h_\delta]$. On the other hand we attain $a^{-1-k} (\log a)^{\tilde{\rho}} \le 1$ for $\delta = 0$ and  $h \in (0,h_\delta]$. So, overall,
\begin{equation*}
2(1 + \gamma) (\varphi')^2 \mathcal{W} |\Phi|^2 \min \big(\mathcal{W}, \frac{h}{4\varphi'} \big)  \le C\tau h^{1 - \sigma}.
\end{equation*} 

Continuing on, we decrease $h_\delta$ if necessary, so $V_0(x) = 0$ if $|x| = r > a$. Complementing this with $\mathcal{W} \le r \mathcal{G}/2$, $|V_S| \le c_Sm_S(r)r^{-1 -\delta}$, $\min(\mathcal{W}, h/(4 \varphi')) \le h/(4 \varphi')$, $\varphi' = \tau h^{-\sigma}\varphi'_0$, $\tau \ge 1$, the lower bound in \eqref{est phi0 prime r large}, and $\sigma = 1/3$,

  \begin{equation*}
 2(1 + \gamma) h^{-2} \mathcal{W} |V_S|^2 \min \big(\mathcal{W}, \frac{h}{4\varphi'} \big) \le C h^{-1+\sigma} \frac{\mathcal{G} m_S^2}{r^{1 + 2\delta}\varphi'_0} \le C h^{-2\sigma} \frac{(m_S\mathcal{G})^2}{\mathcal{G}(a) r^{2\delta}} a^{-1 + k}
 \end{equation*}
 Recalling $m_S \le 1$ and $\mathcal{G} = r^\eta$ for $0 < \delta \le 1$, while $m_S = (\log(r) +1)^{-\rho}$ and $\mathcal{G} = (\log r)^{\tilde{\rho}}$ for $\delta =0$, we substitute,
 \begin{equation*}
 \frac{(m_S\mathcal{G})^2}{\mathcal{G}(a) r^{2\delta}}  \begin{cases} \le a^{-\eta} r^{2\eta - 2\delta} \le a^{-2\delta + \eta} &0 < \delta \le 1 \\
= (\log r)^{-\tilde{\rho}} ( (\log r)^{\tilde{\rho}} (\log(r) +1)^{-\rho})^2 r^{-2\delta} \le (\log a)^{-\tilde{\rho}} & \delta = 0 \\
 \end{cases}.
 \end{equation*}
 Here, we also used that, when $0 < \delta \le 1$, $\delta \ge \eta$ (see \eqref{implication}), and when $\delta = 0$, $\tilde{\rho} \le \rho$. Combining this with the previous estimate, the second line of \eqref{unpack a delta posit}, and the second line of \eqref{unpack a delta 0}, 
 \begin{equation*}
 \begin{gathered}
 2(1 + \gamma) h^{-2} \mathcal{W} |V_S|^2 \min \big(\mathcal{W}, \frac{h}{4\varphi'} \big) \\
 \le \begin{cases} C h^{ - 2\sigma} a^{-1 -2\delta + k + \eta} \le C \log(h^{-1})^{-\frac{T}{3}}  & 0 < \delta \le 1\\[6pt] 
 C h^{ - 2\sigma} a^{-1 +k} (\log a)^{-\tilde{\rho}} \le e^{-\frac{2}{3}t} & \delta = 0 
   \end{cases}.
   \end{gathered}
   \end{equation*}
The final term we need to estimate is,   
   \begin{equation*}
 V_L + \mathcal{W}\big(V'_L + \frac{h^2q}{4r^2} \big) \le \frac{3E}{8} + c_Ly\mathbf{1}_{1 < r < b} +\frac{h^2}{2r^3} ,
 \end{equation*}
where we used that, for $r > a$, $0 \le q \le 2/r$, $V_L \le (E_{\min}/8) + c_Ly \mathbf{1}_{1 < r \le b}$ (see \eqref{b}), and
\begin{equation*}
\mathcal{W}V'_L \le \frac{r}{2} \frac{E}{2c_L m_L} \frac{c_L m_L}{r} = \frac{E}{4}.
\end{equation*} 
From the above, and taking $h_\delta$ smaller so that $a > b$, we conclude, 
\begin{equation} \label{prelim key est big r}
\begin{gathered}
\begin{split}
A&- (1 + \gamma)B  \\
& \ge \begin{cases} 
 w' \Big[ \frac{5E}{8} - C\big( \tau^2 (\log(h^{-1}))^{2\e_1- \frac{T}{3}} + \tau h^{1 - \sigma} + h^2\big) \Big] & 0 < \delta \le 1 \\[6pt]
 w'\Big[ \frac{5E}{8} - C \big( \tau^2 t^{\tilde{\rho}}e^{-\frac{2}{3}t} + \tau h^{1 -\sigma}  + h^2\big) \Big] & \delta = 0
\end{cases}, 
\end{split} \\
E \in [E_{\min}, E_{\max}], \, h \in (0, h_\delta], \, r > a.
\end{gathered}
\end{equation}

 If $0 < \delta \le 1$, fix $T = 9\e_1$, and further decrease $h_\delta$ as needed, in particular so that $C (\tau^2 (\log(h^{-1}))^{-\e_1} + \tau h^{1 - \sigma} + h^2)\le E_{\min}/8$, to arrive at \eqref{e:keyLower} for $r > a$.
If $\delta =0$, pick $t$ large enough so that $Ct^{\tilde{\rho}}e^{-\frac{2}{3}t} \le E_{\min}/16$, and then further decrease $h_\delta$ so that $C( \tau h^{1- \sigma} + h^2) \le E_{\min}/16$, to attain \eqref{e:keyLower} for $r > a$. 

Thus, by subsections \ref{small r section} and \ref{large r section}, we have demonstrated \eqref{e:keyLower}.

\subsection{Bounding the phase} \label{bd phase section} \hfill \\

Our remaining goal is to show \eqref{e:phiBoundMe}. Recall \eqref{e:phiBounds2} and \eqref{est phi0 prime r large}:
\begin{equation*}
\begin{gathered}
0 \le \varphi_0'(r)\leq \begin{cases} \frac{C}{r^k}& \delta = 0 \text{ or } 1\\[6pt]
   \frac{C(\log(h^{-1}))^{\e_1}}{r^k} &0 < \delta < 1 \end{cases}, \qquad h \in (0, h_\delta], \, 1 < r \le a, \\
0 \le  \varphi_0'(r)\leq \begin{cases}   \frac{a^{1-k}\mathcal{G}(a)}{r\mathcal{G}(r)}& \delta = 0 \text{ or } 1 \\[6pt]
  \frac{C(\log(h^{-1}))^{\e_1}a^{1-k}\mathcal{G}(a)}{r\mathcal{G}(r)}&0 < \delta < 1\end{cases}, \qquad h \in (0, h_\delta], \, r > a.\\
\end{gathered}
\end{equation*}
We also have, from $a = h^{-M}$, $\mathcal{G} = r^\eta$ for $0 < \delta \le 1$, and $\mathcal{G} = (\log r)^{\tilde{\rho}}$ for $\delta =0$,
\begin{equation*}
 \int_{a}^\infty \frac{\mathcal{G}(a)}{s\mathcal{G}(s)}ds = \begin{cases} \eta^{-1} = \log(h^{-1}) & \delta > 0, \\ 
   (1 - \tilde{\rho})^{-1} \log a = M (1 - \tilde{\rho})^{-1} \log(h^{-1}) & \delta = 0.  \end{cases}
\end{equation*}
Using these, we estimate $\varphi_0(r)$,

\begin{equation}
\begin{split} \label{e:phiMax}
\varphi_0(r) &\le  \int_0^{1} \varphi'_0(s) ds +  \int_1^{a} \varphi'_0(s) ds + \int_a^{^\infty} \varphi'_0(s) ds \\
&\le \begin{cases}
C a^{1-k} \log(h^{-1}) & \delta = 1\\
C a^{1-k} \log(h^{-1})^{1 + \e_1} & 0 < \delta < 1\\
C M a^{1-k} \log(h^{-1}) & \delta = 0
\end{cases}, \qquad h \in (0, h_\delta].
\end{split}
\end{equation}

Recall that we found,
 \begin{equation*} 
 a = \begin{cases} 
  h^{-\frac{\sigma}{k} - T\eta \lambda} =  h^{-\frac{\sigma}{k}} (\log(h^{-1}))^{9\e_1} & 0 < \delta \le 1 \\
   h^{-1- T\eta \lambda - t\eta} = e^{t} h^{-1} (\log(h^{-1}))^T & \delta = 0
  \end{cases},
  \end{equation*}
where we used that we have fixed $T = 9\e_1$ when $0 < \delta \le 1$ and $T = 3 \tilde{\rho}/2$ when $\delta = 0$. Combining this with \eqref{k}:
\begin{equation*}
k = \begin{cases} 1 & \delta = 1
 \\ \frac{1 + 2\delta - \lambda^{-1}}{3} & 0 < \delta < 1 \\
\frac{1}{3} & \delta = 0 \end{cases},
\end{equation*}
and $1/3 \le k \le 1$, we see that 
\begin{equation*}
\begin{split}
&a^{1-k} \\
&= \begin{cases}
1 & \delta =1 \\
h^{-\frac{\sigma}{k}(1-k)} (\log(h^{-1}))^{9(1- k)\e_1}    \le  h^{-\frac{2(1 - \delta) + \lambda^{-1}}{3(1 + 2\delta - \lambda^{-1})}}(\log(h^{-1}))^{6\e_1 } & 0 < \delta < 1 \\
h^{-(1 + \frac{\tilde{\rho}}{2k} \eta \lambda + t \eta)(1 - k)} = e^{\frac{2}{3}t} h^{-\frac{2}{3}} (\log(h^{-1}))^{\tilde{\rho}}& \delta = 0
\end{cases}, \quad h \in (0, h_\delta], \, r >0.
\end{split}
\end{equation*}
In the case $\delta = 0$, we have already fixed $t \ge 1$ independent of $h$. Therefore, for some $C_\delta > 0$ independent of $h$,
\begin{equation*}
\varphi_0(r) \le \begin{cases} C_\delta \log(h^{-1}) & \delta = 1 \\  C_\delta  h^{-\frac{2(1 - \delta) + \lambda^{-1}}{3(1 + 2\delta - \lambda^{-1})}}(\log(h^{-1}))^{1 + 7\e_1 } & 0< \delta < 1 \\
C_\delta h^{-\frac{2}{3}}(\log(h^{-1}))^{1 + \tilde{\rho}} & \delta = 0
  \end{cases}, \qquad h \in (0, h_\delta],\, r> 0.
\end{equation*}
Noting $\e_1 = \e/7$ (see \eqref{ep}), we have arrived at \eqref{e:phiBoundMe}.\\
\end{proof}

\section{Carleman estimate} \label{Carleman estimate section}
Our goal in this section is to prove Lemma \ref{Carleman lemma}, which is a Carleman estimate from which Theorem \ref{Linfty thm} follows.
 
\begin{lemma} \label{Carleman lemma}
Suppose the assumptions of Proposition \ref{p:key} hold. There exist $C > 0$ and $h_\delta \in (0,1)$, both independent of $h$ and $\ep$, so that
\begin{equation} \label{Carleman est}
\|\langle x \rangle^{-\frac{1+ \eta}{2}} e^{\varphi/h} v \|^2_{L^2(\R^n)} \leq
  e^{C/h} \|\langle x \rangle^{\frac{1 + \eta}{2}}e^{\varphi/h}(P(h) - E \pm i\varepsilon)v \|^2_{L^2(\R^n)} 
+   e^{C/h}  \varepsilon \| e^{\varphi/h}v \|^2_{L^2(\R^n)}.
\end{equation}
for all $E \in [E_{\min}, E_{\max}]$, $h \in (0, h_\delta]$, $0 \le \varepsilon \le h$, and  $v \in C_{0}^\infty(\mathbb{R}^n)$.
\end{lemma}

There are three steps to the proof of Lemma \ref{Carleman lemma}. First, by way of Proposition \ref{p:key}, we establish a Carleman estimate which is similar to \eqref{Carleman est} but has a loss at the origin, because the weight $w$ vanishes quadratically as $r \to 0$ (see \eqref{w}). We call this the away-from-origin estimate. Second, we use Obovu's result \cite[Lemma 2.2]{ob23}, which is based on Mellin transform techniques, to obtain an estimate for small $r$ which does not have loss as $r \to 0$. In fact, the pertinent weight in Obovu's estimate is unbounded as $r \to 0$. We call this the near-origin estimate. The third and final step is to glue together the near- and away-from-origin estimates.

\begin{proof}[Proof of Lemma \ref{Carleman lemma}]
We give the proof of Lemma \ref{Carleman lemma} over the course of subsections \ref{afo est section}, \ref{no est section}, and \ref{glue ests section}. The notation $\int_{r,\theta}$ denotes the integral over $(0,\infty) \times \US^{n-1}$ with respect to the measure $dr d\theta$.  Throughout, $C > 0$ and $h_\delta \in (0, 1)$ are constants, both independent of $h$ and $\ep$, whose values may change from line to line.  

\subsection{Away-from-orgin estimate} \label{afo est section} \hfill \\

We begin from \eqref{e:lowerDerivative2}. Applying \eqref{e:keyLower}, we bound the right side of \eqref{e:lowerDerivative2} from below. For some $h_\delta \in (0,1)$,

\begin{equation} \label{lower bound wF prime}
\begin{gathered}
    (wF)' \ge -\frac{2(1 + \gamma)w^2}{\gamma h^2w'} \| P^{\pm}_\varphi(h) u \|^2 \mp 2 \varepsilon w \imag \langle u, u' \rangle +  \frac{\gamma}{2(1 + \gamma)} w' \|hu'\|^2 + \frac{E_{\min}}{2}w' \| u\|^2, \\
    E \in [E_{\min}, E_{\max}], \quad h \in (0, h_\delta], \quad r \neq 1, a, \quad u = e^{\varphi/h} r^{(n-1)/2}v.
  \end{gathered}
\end{equation}

Next, we integrate both sides of \eqref{lower bound wF prime}. We integrate $\int^\infty_0dr$ and use $wF, \, (wF)' \in L^1((0,\infty);dr)$, and $wF(0) = wF(\infty) = 0$, hence $\int_{0}^\infty (wF)'dr = 0$. Using also \eqref{w squared over w prime} yields
\begin{equation} \label{penult est}
\begin{gathered}
    \int_{r,\theta} w'\left(|u|^2 +|hu'|^2 \right) \leq  e^{C/h} \int_{r,\theta} (1 + r)^{1 + \eta}|P^{\pm}_\varphi(h)u|^2  + \frac{\ep}{h} \int_{r, \theta} w(|u|^2 + |hu'|^2), \\
    E \in [E_{\min}, E_{\max}], \quad h \in (0, h_\delta].
 \end{gathered}
\end{equation}

The remaining task is to absorb the term involving $u'$ on the right side of \eqref{penult est} into the left side. To this end, let $\psi (r) \in C^\infty([0, \infty); [0,1])$ with $\psi = 0$ near $[0, 1/2]$, and $\psi = 1$ near $[1, \infty)$. We have
\begin{equation} \label{split u prime int}
\begin{split}
 \frac{\ep}{h} \int_{r, \theta} w |hu'|^2 &\le  \frac{\ep}{h} \int_{0 < r < 1, \theta} w |hu'|^2 + \frac{\ep}{h} \int_{r, \theta} w |h(\psi u)'|^2 \\
  &\le \frac{1}{2} \int_{0 < r < 1, \theta} w' |hu'|^2 + \frac{\ep}{h} \int_{r, \theta} w |h(\psi u)'|^2.
 \end{split}
\end{equation}
To get the second line, we used $\ep \le h$ and $2w \le  w'$ for $0 < r < 1$, see \eqref{w}. The first term in the second line of \eqref{split u prime int} is easily absorbed into the left side of \eqref{penult est}. As for the second term, integrating by parts,

\begin{equation*} %\label{rewrite Pu bar u}
    \begin{split}
        \real \int_{r,\theta}(P^{\pm}_\varphi(h) (\psi u))\overline{\psi u} &= \int_{r,\theta} |h(\psi u)'|^2  + \real \int_{r,\theta} 2 h\varphi' (\psi u)' \overline{\psi u} + \int_{r, \theta} (h^2 r^{-2} \Lambda (\psi u))\overline{\psi u} \\
        &+ \int_{r,\theta} h\varphi''|\psi u|^2 + \int_{r,\theta} \left(V + E - (\varphi')^2 \right)
        |\psi u|^2,\\
    \end{split}
\end{equation*}
and 
\begin{equation*} %\label{int by parts}
    \int_{r, \theta} h \varphi''|\psi u|^2 = - \real \int_{r,\theta} 2 h\varphi' (\psi u)' \overline{\psi u}. 
\end{equation*}
These two identities, together with the facts that  $\ep \le h$, $\Lambda \ge -1/4$, $r^{-2}$ is bounded on $\supp \psi$, $w' = 2r \ge 1$ on $\supp \psi'$, and $|V + E - (\varphi')^2| \le e^{C/h}$ on $\supp \psi$ for $E \in [E_{\min}, E_{\max}]$ and $h \in (0,h_\delta]$, imply
\begin{equation} \label{handle deriv u term}
\begin{split}
 \frac{\ep}{h}  \int_{r, \theta} w |h(\psi u)'|^2 &\le \ep e^{C/h} \int_{r,\theta} |u|^2 +  C \int_{r, \theta}(r +1)^{1 + \eta} |P^{\pm}_\varphi(h)u|^2\\
 &+ Ch^2 \int_{1/2 < r < 1, \theta} w'|hu'|^2, \qquad  E \in [E_{\min}, E_{\max}],\, h \in (0,h_\delta], \, 0 < \ep \le h.
    \end{split}
\end{equation}

For $h$ sufficiently small, the second line of \eqref{handle deriv u term} is readily absorbed into the right side of \eqref{penult est}. Therefore, \eqref{penult est}, \eqref{split u prime int}, and \eqref{handle deriv u term} imply 
\begin{equation} \label{final est}
\begin{gathered}
    \int_{r,\theta} w'(|u|^2 + |hu'|^2) \le e^{C/h} \int_{r,\theta} (r + 1)^{1 + \eta}|P^{\pm}_\varphi(h)u|^2  + \varepsilon  e^{C/h}  \int_{r, \theta} |u|^2,\\
     E \in [E_{\min}, E_{\max}], \quad h \in (0, h_\delta], \quad 0 < \ep \le h.
 \end{gathered}
\end{equation}

\subsection{Statement of the near-origin estimate} \label{no est section}
\begin{lemma}[{\cite[Lemma 2.2]{ob23}}]  
Fix $t_0 \in (-1/2, 0)$. There exist $C > 0$ and $\alpha_\beta, \, h_\delta \in (0,1)$, all independent of $\ep$ and $h$, so that for all  $E \in [E_{\min}, E_{\max}]$, $h \in (0, h_\delta]$, $0 \le \ep \le 1$, and $v \in C^\infty_0(\R^n)$,
\begin{equation} \label{obovu est}
\begin{split}
\int_{0 < r < 1/2, \theta} |r^{-\frac{1}{2} - t_0} r^{\frac{n-1}{2}} v |^2 &\le Ch^{-4} \big( \int_{0 < r < 1, \theta} |r^{\frac{3}{2} - t_0} r^{\frac{n-1}{2}} (P - E \pm i\ep) v |^2 \\
&+ \int_{\alpha < r < 1, \theta} |r^{\frac{3}{2} - t_0} (V - E \pm i\ep) r^{\frac{n-1}{2}}v |^2 \\
&+ h^2  \int_{1/2 < r < 1, \theta} |r^{\frac{3}{2} - t_0} r^{\frac{n-1}{2}} v  |^2 + h  \int_{1/2 < r < 1, \theta} |r^{\frac{3}{2} - t_0} h(r^{\frac{n-1}{2}} v)'   |^2 \big),
\end{split}
\end{equation}
where 
\begin{equation} \label{alpha}
\alpha \defeq \alpha_\beta h^{\frac{2}{2 - \beta}}.
\end{equation}
\end{lemma}

\subsection{Combining the near- and away-from-origin estimates} \label{glue ests section} \hfill \\
For $v \in C^\infty_0(\R^n)$, set $\tilde{u} = r^{\frac{n-1}{2}} v$. We have
\begin{equation} \label{begin gluing strategy}
\begin{split}
\| \langle r \rangle^{-\frac{1 + \eta}{2}} v \|^2_{L^2} &=  \int_{0 < r < \alpha, \theta} |\langle r \rangle^{-\frac{1 + \eta}{2}} \tilde{u}|^2 + \int_{\alpha < r < a, \theta} |\langle r \rangle^{-\frac{1 + \eta}{2}} \tilde{u}|^2 + \int_{r  > a, \theta} |\langle r \rangle^{-\frac{1 + \eta}{2}} \tilde{u}|^2 \\
&\le  \alpha^{1 + 2t_0} \int_{0 < r < \alpha, \theta} |r^{-\frac{1}{2} - t_0} \tilde{u}|^2+ \alpha^{-1} \int_{\alpha < r < a, \theta} w'|\tilde{u}|^2 + \log(h^{-1})^C \int_{r > a, \theta} w' |\tilde{u}|^2,
\end{split}
\end{equation}
where we used \eqref{univ lwr bd w prime} and
\begin{equation} \label{w prime 2r}
 w' = 2r, \qquad 0 < r < a.
\end{equation}
Furthermore,
\begin{gather}
 \int_{0 < r < 1, \theta} |r^{\frac{3}{2} - t_0} r^{\frac{n-1}{2}} (P - E \pm i\ep) v |^2 \le \int_{r , \theta} |\langle r \rangle^{\frac{1 + \eta}{2}} r^{\frac{n-1}{2}} (P - E \pm i\ep) v |^2, \label{handle P straightforwardly} \\
 \int_{\alpha < r < 1, \theta} |r^{\frac{3}{2} - t_0} (V - E \pm i\ep) \tilde{u} |^2 \le \alpha^{2 - 2t_0 - 2 \beta} \int_{r , \theta} w' |\tilde{u} |^2, \label{convert V to powers of alpha} \\
  \int_{1/2< r < 1, \theta} |r^{\frac{3}{2} - t_0} \tilde{u} |^2 +   \int_{1/2 < r < 1, \theta} |r^{\frac{3}{2} - t_0} h\tilde{u}' |^2 \le e^{C/h} \int_{r, \theta} w'( |u |^2 + |h u '|^2), \label{est commutator terms}
 \end{gather}
where, as in subsection \ref{afo est section}, $u = e^{\varphi/h} r^{(n-1)/2} v = e^{\varphi/h} \tilde{u} $. To get \eqref{convert V to powers of alpha}, we used \eqref{V0} and \eqref{w prime 2r}. To get \eqref{est commutator terms}, we used \eqref{w prime 2r}, \eqref{phi prime}, and \eqref{tau}, hence
\begin{equation*}
| r^{\frac{1}{2} - t_0} h\tilde{u}'|^2 = |r^{\frac{1}{2} - t_0}(e^{-\frac{\varphi}{h}} hu' -  \varphi'e^{-\frac{\varphi}{h}} u)|^2 \le C w'( |u |^2 + \max \varphi' |h u '|^2) \le e^{C/h} w'( |u |^2 + |h u '|^2).
\end{equation*}

Consider now the second line of \eqref{begin gluing strategy}.  We bound the first term appearing there using \eqref{obovu est} and \eqref{handle P straightforwardly} through \eqref{est commutator terms} ($\alpha \le 1/2$ for $h$ small enough, see \eqref{alpha}). We bound the second and third terms using \eqref{final est}. Since negative powers of $\alpha$ are bounded from above by $e^{C/h}$ for $h$ small, we conclude, for $E \in [E_{\min}, E_{\max}], \, h \in (0, h_\delta],$ and $0 < \ep \le h$,
\begin{equation} \label{penult glue step}
\begin{split}
\| \langle r \rangle^{-\frac{1 + \eta}{2}} v \|^2_{L^2} &\le  e^{C/h} \big( \int_{r , \theta} |\langle r \rangle^{\frac{1 + \eta}{2}} r^{\frac{n-1}{2}} (P - E \pm i\ep) v |^2 +  \int_{r, \theta} w'( |u|^2 + |hu'|^2)  \\  
&+ e^{C/h} \int_{r,\theta} (r + 1)^{1 + \eta}|P^{\pm}_\varphi(h)u|^2  + \varepsilon  e^{C/h}  \int_{r, \theta} |u|^2 \\
&\le e^{C/h} \|\langle r \rangle^{\frac{1 + \eta}{2}}e^{\varphi/h}(P(h) - E \pm i\varepsilon)v \|^2_{L^2} 
+   e^{C/h}  \varepsilon \| e^{\varphi/h}v \|^2_{L^2} \\
&+ e^{C/h}  \int_{r, \theta} w'( |u|^2 + |hu'|^2),
\end{split} 
\end{equation}
where we have used 
\begin{equation*} 
2^{-\frac{1 + \eta}{2}} \le \left( \frac{\langle r \rangle}{r+1} \right)^{1 + \eta}.
\end{equation*}
Employing \eqref{final est} once more, to bound the last line of \eqref{penult glue step}, we arrive at \eqref{Carleman est} as desired.\\
\end{proof}

\section{Resolvent estimates} \label{resolv est section}
In this section, we deduce Theorem \ref{Linfty thm} from the Carleman estimate \eqref{Carleman est}. This same argument has been presented before, see, e.g., \cite{da14, gash22a, gash22b, ob23}.  But we include it here for the sake of completeness. The constants $C> 0$ and $h_\delta \in (0,1)$ may change between lines but stay independent of $E$, $\ep$, and $h$.

\begin{proof}[Proof of Theorem~\ref{Linfty thm}]

By the spectral theorem for self-adjoint operators, the bounds \eqref{Linfty resolv est delta 1}, \eqref{Linfty resolv est delta between 0 1}, and \eqref{Linfty resolv est delta 0} clearly hold for $\ep > h$. Therefore, to prove Theorem \ref{Linfty thm}, it suffices to consider $0 < \ep \le h$. 

Since increasing $s$ in \eqref{defn g} decreases the operator norm, to establish a certain estimate for \eqref{defn g} for fixed $s > 1/2$ independent $h$, it suffices to show the same estimate for $h$ small enough and an $h$-dependent $s$ of the form $(1+ \eta)/2 < 1$. For the rest of the proof, we assume $s$ has this form.

By Lemma \ref{Carleman lemma},
\begin{equation} \label{mult through by exp}
e^{-C_\varphi/h}\|\langle x \rangle^{-s} v \|^2_{L^2} \leq
  e^{C/h} \|\langle x \rangle^{s} (P(h) - E \pm i\varepsilon)v \|^2_{L^2} 
+  \varepsilon e^{C/h} \| v \|^2_{L^2},
\end{equation}
for all $E \in [E_{\min}, E_{\max}]$, $h \in (0, h_\delta]$, $0 \le \ep \le h$, and $v \in C^\infty_0(\R^n)$, and where $C_\varphi = C_\varphi(h) \defeq 2 \max \varphi$.  Moreover, for any $\gamma > 0$,
\begin{equation} \label{epsilon v}
\begin{split}
2\varepsilon \| v \|^2_{L^2} &= -2 \imag\langle (P(h) - E \pm i\varepsilon)v, v \rangle_{L^2} 
\\& \le  \gamma^{-1}\|\langle x \rangle^{s}(P(h) - E \pm i \varepsilon)v \|^2_{L^2} 
+ \gamma\|\langle x \rangle^{-s} v\|^2_{L^2}.  
\end{split}
\end{equation}
 Setting $\gamma =  e^{-(C+C_\varphi)/ h} $, and using \eqref{epsilon v} to estimate $\varepsilon \| v \|^2_{L^2}$ from above in \eqref{mult through by exp}, we absorb the $\| \langle x\rangle^{-s} v\|_{L^2}$ term that now appears on the right of \eqref{mult through by exp} into the left side. Multiplying through by $2e^{C_\varphi/h}$, and applying \eqref{e:phiBoundMe}, we find, for $E \in [E_{\min}, E_{\max}]$, $h \in (0, h_\delta]$,  $0 \le \ep \le h$, and $v \in C^\infty_0(\R^n)$,
\begin{equation} \label{penult}
\begin{split}
 \|&\langle x \rangle^{-s} v \|_{L^2}^2
 \\& \le \begin{cases}
\exp\big( C h^{-4/3} (\log(h^{-1}) \big) \|\langle x \rangle^s(P(h) - E \pm i \varepsilon)v \|_{L^2}^2  & \delta = 1, \\
   \exp\big( C h^{- \frac{4}{3} - \frac{2(1 - \delta) + \lambda^{-1}}{3(1 + 2\delta - \lambda^{-1})}} (\log(h^{-1}))^{1 + \epsilon} \big)  \|\langle x \rangle^s(P(h) - E \pm i \varepsilon)v \|_{L^2}^2 & 0 < \delta < 1, \\
 \exp\big(Ch^{-2}(\log(h^{-1}))^{1 + \tilde{\rho}} \big)  \|\langle x \rangle^s(P(h) - E \pm i \varepsilon)v \|_{L^2}^2 & \delta = 0.
   \end{cases}
   \end{split}
   \end{equation}
The final task is to use \eqref{penult} to show, for $E \in [E_{\min}, E_{\max}]$, $h \in (0,h_\delta]$, $0 < \ep \le h$, and $f \in L^2(\R^n)$,
\begin{equation} \label{ult}
\begin{split}
\|\langle x \rangle^{-s}(P(h)-E \pm & i\varepsilon)^{-1} \langle x \rangle^{-s} f \|^2_{L^2}\\
&\le \begin{cases} 
\exp\big( C h^{-4/3} (\log(h^{-1}) \big) h) \|f\|^2_{L^2}  & \delta = 1, \\
   \exp \big( C h^{- \frac{4}{3} - \frac{2(1 - \delta) + \lambda^{-1}}{3(1 + 2\delta - \lambda^{-1})}} (\log(h^{-1}))^{1 + \epsilon} \big) \|f\|^2_{L^2} & 0 < \delta < 1, \\
 \exp\big(Ch^{-2}(\log(h^{-1}))^{1 + \tilde{\rho}} \big)  \|f\|^2_{L^2} & \delta = 0.
   \end{cases}
\end{split}
\end{equation}
    from which Theorem~\ref{Linfty thm} follows. To establish \eqref{ult}, we prove a simple Sobolev space estimate  and then apply a density argument that relies on \eqref{penult}. 

The operator
\begin{equation*}
[P(h), \langle x \rangle^s]\langle x \rangle^{-s} = \left(-h^2 \Delta \langle x \rangle^s - 2h^2 (\nabla \langle x \rangle^s) \cdot \nabla \right) \langle x \rangle^{-s}
\end{equation*}
is bounded $H^2(\R^n) \to L^2(\R^n)$. So, for $v \in H^2(\R^n)$ such that $\langle x \rangle^s v \in H^2(\R^n)$,
 \begin{equation}\label{Ceph}
 \begin{split}
\|\langle x \rangle^{s}(P(h)-E \pm i \varepsilon)v\|_{L^2}  &\le \|(P(h)-E \pm i \varepsilon)\langle x \rangle^{s} v \|_{L^2} +  \|[P(h),\langle x \rangle^{s}]\langle x \rangle^{-s}\langle x \rangle^{s}v \|_{L^2}
\\& \le C_{E_{\max}, \varepsilon, h} \| \langle x \rangle^{s}v \|_{H^2},
\end{split} 
\end{equation}
for some constant $C_{E_{\max}, \varepsilon, h}>0$ depending on $E_{\max}$, $\varepsilon$ and $h$.

Given $f \in L^2(\R^n)$, the function $ u= \langle x \rangle^{s}(P(h)-E\pm i\varepsilon)^{-1}\langle x \rangle^{-s} f \in H^2(\R^n)$ because 
\begin{equation*}
u = (P(h) - E \pm i \varepsilon)^{-1} (f - w), \qquad w = \langle x \rangle^{s} [P(h), \langle x \rangle^{-s}]   \langle x \rangle^{s}  \langle x \rangle^{-s}u,
\end{equation*}
with  $\langle x \rangle^{s} [P(h), \langle x \rangle^{-s}]   \langle x \rangle^{s}$ being bounded $L^2(\R^n) \to L^2(\R^n)$ since $s < 1$.

Now, choose a sequence $v_k \in C_{0}^\infty$ such that $ v_k \to  \langle x \rangle^{s}(P(h)-E \pm i\varepsilon)^{-1}\langle x \rangle^{-s} f$ in $H^2(\R^n)$. Define $\tilde{v}_k \defeq \langle x \rangle^{-s}v_k$. Then, as $k \to \infty$,
\begin{equation*}
\begin{split}
\| \langle x \rangle^{-s} \tilde{v}_k - \langle x \rangle^{-s} (&P(h)-E\pm i \varepsilon)^{-1}\langle x \rangle^{-s}f \|_{L^2}  \\
&\le \| v_k - \langle x \rangle^{s} (P(h)-E\pm i \varepsilon)^{-1}\langle x \rangle^{-s}f \|_{H^2} \to 0.
\end{split}
\end{equation*}
Also, applying \eqref{Ceph},
\begin{equation*}
\|\langle x \rangle^{s}(P(h)-E\pm i \varepsilon)\tilde v_k - f\|_{L^2} \le C_{E_{\max}, \varepsilon,h} \|v_k - \langle x \rangle^{s} (P(h)-E \pm i \varepsilon)^{-1} \langle x \rangle^{-s} f \|_{H^2} \to 0.
\end{equation*} 
We then achieve \eqref{ult} by replacing $v$ by $\tilde{v}_k$ in \eqref{penult} and sending $k \to \infty$.\\
\end{proof}

\end{document}